\documentclass[12pt, reqno]{amsart}
\usepackage{amsmath}
\usepackage{amsthm}
\usepackage{amssymb}
\usepackage{amsrefs}
\usepackage{mathrsfs}
\usepackage[usenames]{color}
\usepackage{ifthen}
\usepackage{graphicx}
\usepackage[all]{xy}
\usepackage{xcolor}
\usepackage{bm}
\usepackage{enumerate}

\newtheorem{thm}{}[section]
\newtheorem{theorem}[thm]{Theorem}
\newtheorem{corollary}[thm]{Corollary}
\newtheorem{lemma}[thm]{Lemma}
\newtheorem{proposition}[thm]{Proposition}

\theoremstyle{remark}

\newtheorem{question}[thm]{Question}
\newtheorem{example}[thm]{Example}

\numberwithin{equation}{section}
\allowdisplaybreaks

\newcommand{\normiii}[1]{{\left\vert\mkern-1.8mu\left\vert\mkern-1.8mu\left\vert #1 \right\vert\mkern-1.8mu\right\vert\mkern-1.8mu\right\vert}}

\DeclareMathOperator{\core}{core}
\DeclareMathOperator{\spn}{span}
\DeclareMathOperator{\conv}{co}
\DeclareMathOperator{\Ker}{Ker}

\newcommand{\DD}{\ensuremath{\mathbb{D}}}
\newcommand{\RR}{\ensuremath{\mathbb{R}}}
\newcommand{\F}{\ensuremath{\mathcal{F}}}
\newcommand{\II}{\ensuremath{\mathcal{I}}}
\newcommand{\Fou}{\ensuremath{\mathcal{F}}}
\newcommand{\LL}{\ensuremath{\mathcal{L}}}
\newcommand{\NN}{\ensuremath{\mathbb{N}}}
\newcommand{\FF}{\ensuremath{\mathbb{F}}}
\newcommand{\xx}{\ensuremath{\bm{x}}}
\newcommand{\Id}{\ensuremath{\mathrm{Id}}}
\newcommand{\ee}{\ensuremath{\bm e}}

\begin{document}

\title[On a `philosophical' question about Banach envelopes]{On a `philosophical' question about Banach envelopes}
\author[F. Albiac]{Fernando Albiac}
\address{Mathematics Department-InaMat\\
Universidad P\'ublica de Navarra\\
Campus de Arrosad\'{i}a\\
Pamplona\\
31006 Spain}
\email{fernando.albiac@unavarra.es}

\author[J.~L. Ansorena]{Jos\'e L. Ansorena}
\address{Department of Mathematics and Computer Sciences\\
Universidad de La Rioja\\
Logro\~no\\
26004 Spain}
\email{joseluis.ansorena@unirioja.es}

\author[P. Wojtaszczyk]{Przemys\l{}aw Wojtaszczyk}
\address{Institute of Mathematics of the Polish Academy of Sciences\\
00-656 Warszawa\\
ul. \'Sniadeckich 8\\
Poland}
\email{wojtaszczyk@impan.pl}

\begin{abstract}
We show how to construct nonlocally convex quasi-Banach spaces $X$ whose dual separates the points of a dense subspace of $X$ but does not separate the points of $X$. Our examples connect with a question raised by Pietsch [\emph{About the Banach envelope of $l_{1,\infty}$}, Rev. Mat. Complut. 22 (2009), 209--226] and shed light into the unexplored class of quasi-Banach spaces with nontrivial dual which do not have sufficiently many functionals to separate the points of the space.
\end{abstract}

\thanks{F. Albiac acknowledges the support of the Spanish Ministry for Economy and Competitivity under Grant MTM2016-76808-P for \emph{Operators, lattices, and structure of Banach spaces}. F. Albiac and J.~L. Ansorena acknowledge the support of the Spanish Ministry for Science, Innovation, and Universities under Grant PGC2018-095366-B-I00 for \emph{An\'alisis Vectorial, Multilineal y Aproximaci\'on}. P. Wojtaszczyk was supported by National Science Centre, Poland grant UMO-2016/21/B/ST1/00241.}

\maketitle

\section{Introduction and background}\label{SectionQB}
\noindent The \emph{Banach envelope} of a quasi-Banach space 
 $X$ over the real or complex field $\FF$, consists of a Banach space $\widehat{X}$ together with a norm-one linear map $J_X\colon X \to \widehat{X}$ satisfying the following property: for every Banach space $Y$ and every continuous linear map $T\colon X \to Y$ there is a unique continuous linear map $\widehat{T}\colon \widehat{X}\to Y$ such that $\widehat{T}\circ J_X=T$, and the ``extension'' $\widehat{T}$ has a norm bounded by that of $T$. Pictorially we have
\[\xymatrix{
 \widehat{X}\ar[drr]^{\widehat{T}} & \\
X \ar[u]^{J_X} \ar[rr]_T && Y
}\]
In particular, $X$ and $\widehat{X}$ have the same dual space. We will refer to $J_{X}$ as the \emph{envelope map} of $X$.

Since the Banach envelope of $X$ is defined by means of a universal property, it is unique in the sense that if a Banach space $Z$ and a bounded linear map $J\colon X\to Z$ satisfy the property, then there is an linear isometry $S\colon \widehat{X} \to Z$ with $S\circ J_X=J$. In general, if a Banach space $Z$ and a bounded linear map $J\colon X\to Z$ are such that $\widehat{J}\colon \widehat{X}\to Z$ is a $\lambda$-isomorphism we say that $Z$ is an \emph{$\lambda$-isomorphic representation} of the Banach envelope of $X$ via $J$. When the constant $\lambda\ge 1$ is irrelevant, we simply drop it. A $1$-isomorphic representation will be called \emph{isometric representation}.

A construction of the Banach envelope of a quasi-Banach space $X$ goes as follows. Let $\conv(B_{X})$ be the convex hull of the closed unit ball of $X$, and denote by $\Vert \cdot \Vert_{e}$ the Minkowski functional of $\conv(B_{X})$, i.e.,
\begin{equation}\label{BEnorm}
\Vert x \Vert_{e}= \inf\left\{t>0\colon x/t \in \conv(B_{X}) \right\},\quad x\in X.
\end{equation}
Equation \eqref{BEnorm} yields a semi-norm on $X$.
Moreover, it is easy to prove that,
 if $\Vert \cdot\Vert$ denotes de quasi-norm on $X$,
 $\Vert \cdot \Vert_{e}$ is also given by the formula
\[
\Vert x \Vert_{e}=\inf\left\{\sum_{i=1}^{n}\Vert x_{i}\Vert\colon \sum_{i=1}^{n} x_{i}=x,\, x_{i}\in X,\, n\in \NN\right\}, \quad x\in X,
\]
which makes $\Vert\cdot\Vert_{e}$ the largest seminorm $\normiii{\cdot}$ on $X$ for which $\normiii{x}\le \Vert x\Vert$ for all $x\in X$. Passing to the quotient space $X/N$, with
\begin{equation}\label{eq:DefN}
N=\{x\in X\colon \Vert x\Vert_{e}=0\},
\end{equation}
we obtain a normed linear space whose completion, $\widehat X$, is the required Banach envelope, with associated envelope map $J_X:=I \circ Q$, where $Q\colon X\to X/N$ is the canonical quotient map and
$
I\colon X/N\to \widehat{X}
$
is the inclusion map. Note that, by construction, $\Ker(J_X)=N$ and $J_X(X)$ is dense in $\widehat X$. Note also that by the Hahn-Banach theorem,
\begin{align*}
N&=(X^*)^\perp :=\{x \in X \colon x^*(x)=0  \text{ for all } x^*\in X^*\}\\
&=\{ x \in X \colon T(x)=0 \text{ for all Banach spaces $Z$ and all }  T\in\LL(X,Z)\},
\end{align*}
where $\LL(X,Z)$ denotes the quasi-Banach space of all bounded linear operators from $X$ to $Z$. This identity yields that the existence of nonzero elements in $N$ can be equivalently re-stated in terms of the injectivity of the map $J_{X}$ as follows.
\begin{lemma}\label{lem:zero} Given a quasi-Banach space $X$ and a subspace $X_0$ of $X$, the following are equivalent.
\begin{enumerate}[(i)]
\item The semi-norm $\Vert \cdot\Vert_{e}$ given by \eqref{BEnorm} is a norm on $X_0$.
\item The envelope map $J_X\colon X\to \widehat X$ restricted to $X_0$ is one-to-one.
\item There is a Banach space $Z$ and a bounded linear operator from $X$ into $Z$ whose restriction to $X_0$ is one-to-one.
\item $X_0\cap (X^*)^\perp=\{0\}$, i.e., the dual space $X^*$ separates the points of $X_0$.
\end{enumerate}
\end{lemma}

In this paper we will refer to $X$ as having the \emph{point-separation property} if $ (X^*)^\perp=\{0\}$, i.e., the dual space $X^*$ separates the points of $X$ (\cite{KPR1984}). Applying Lemma~\ref{lem:zero} with $X=X_0$ we obtain the following.
\begin{lemma}\label{lem:2}
Let $X$ be a quasi-Banach space.
The following are equivalent.
\begin{enumerate}[(i)]
\item The envelope map $J_{X}$ is one-to-one.
\item There is a one-to-one bounded linear map from $X$ into a Banach space.
\item $X$ has the point-separation property.
\end{enumerate}
\end{lemma}
Let us also bring up the connection between the existence of finite-dimensional uncomplemented subspaces and the existence of nonzero vectors in the space $N$ defined as in \eqref{eq:DefN}.

\begin{lemma}\label{lem:FXUncomplemented}
Let $X$ be a quasi-Banach space. We have
\begin{equation*}
X\setminus (X^*)^\perp=\{ x\in X \colon \FF x \text{ is complemented in } X\}.
\end{equation*}
\end{lemma}

Another way of introducing the Banach envelope of a quasi-Banach space $X$ is by using the bidual. In this case the envelope map is the (not necessarily one-to-one) bidual map $h_X\colon X\to X^{**}$, defined as usual by $h_{X}(x)(f)=f(x)$, and $\widehat X$ is the closure of $h_X(X)$ in $X^{**}$. The more constructive approach we sketched above versus this functional approach has the advantage of inspiring useful generalizations of the notion of Banach envelope. For instance, we could define the $p$-Banach envelope of $X$ for $0<p<1$ similarly to the case $p=1$ (see, e.g., \cite{AACD2019} for a precise definition, and \cites{AlbiacAnsorena2015,AlbiacAnsorena2017+} for applications of $p$-envelopes).

The concept of Banach envelope was independently introduced in the 1970's by Peetre in \cite{Peetre1974}*{p.125}, and by Shapiro in \cite{Shapiro1977}*{p.~116}. Since then it has been considered by many authors in different spaces and settings. The Banach envelope of the spaces $\ell_{p}$ for $p<1$ is (represented by) $\ell_{1}$ via the inclusion map (see \cite{KPR1984}*{p.~28}), whereas the Banach envelope of $L_{p}$ for $p<1$ is zero \cite{Day1940}. More generally, Lorentz spaces $L_{p,q}$ over $[0,1]$ have trivial dual for $0<p\le 1$ and $0<q\le\infty$ with $(p,q)\neq (1,\infty)$ \cites{Hunt1966,Cwikel1973}. As for Lorentz sequence spaces, a similar argument to the one used in the case of $\ell_p$, yields that for $0<p<1$, $0<q<\infty$, and $0<r\le 1$, the Banach envelopes of the spaces $\ell_{p,q}$, of the separable part of $\ell_{p,\infty}$ (which we denote $\ell_{p,\infty}^{\, 0}$,) and of $\ell_{1,r}$ is $\ell_1$ via the inclusion map (see Theorem~\ref{thm:1BesBasisEnv} below). Identifying the Banach envelope of the non-separable space $\ell_{p,\infty}$ is more challenging. Cwikel \cite{Cwikel1973} admirably solved the problem by proving that the Banach envelope of $\ell_{p,\infty}$ for $0<p<1$ is $\ell_1$ via the inclusion map. Note that, by Theorem~\ref{thm:1BesBasisEnv}, this result yields that the dual of $\ell_{p,\infty}/\ell^{\,0}_{p,\infty}$ is trivial. Combining \cite{AEP1988}*{Theorem 3} with \cite{CRS2007}*{Theorem 2.4.14} yields that for $1<r<\infty$ the Banach envelope of $\ell_{1,r}$ is, via the inclusion map, the Banach space consisting of all sequences $f=(a_n)_{n=1}^\infty\in c_0$ such that
\[
\sum_{n=1}^\infty \frac{n^{r-1}}{\log^r(1+n)} (f^{**}(n))^r<\infty,
\]
where $f^{**}(n)=\max_{ |A|=n} n^{-1} \sum_{k\in A} |a_k|$. Similarly, the Banach envelope of $\ell_{1,\infty}^{\, 0}$ is the space known as the \emph{little Marcinkiewicz space}, consisting of all sequences $f\in c_0$ such that
\[
\lim_{n\to\infty} \frac{ n}{\log(n+1)} f^{**}(n)=0
\]
(see \cite{Markus1964}*{pp.~95--108}, \cite{Mitjagin1964}*{p.~823} and \cite{Arino1987}; see also \cite{AABW2019}*{Example 10.25}). In contrast, no `tangible' representation of the Banach envelope of $\ell_{1,\infty}$ is known (see \cite{Pietsch2009}).

In \cite{DRS1969}, Duren, Romberg, and Shields identified the dual space of the Hardy space $H_{p}(\DD)$ for $0<p<1$ and thus they obtained (a representation of) its Banach envelope (cf.\ \cite{Shapiro1976}). The Banach envelope of $H_{p}(\DD)$ is, via its inclusion into the space of analytic functions, a Bergman space which turns out to be isomorphic again to $\ell_{1}$ (see \cite{Woj1984}*{Theorem 8}). Hence, $\widehat{H_{p}(\DD)}$ can also be represented by the space $\ell_1$ via a suitable mapping, namely, the coefficient transform with respect to the (unique, see \cite{Woj1997}) unconditional basis of the space. In \cite{Kalton1977a}, Kalton gave an isomorphic representation of the Banach envelope of the separable part of nonlocally convex Orlicz spaces over $[0,1]$. Given a function $\alpha\colon I\to (0,\infty)$ on an interval $I$, we define $\widehat\alpha\colon I\to[0,\infty)$ by \[
\widehat\alpha(t)=\inf\left\{ \sum_{j=1}^m \lambda_j \alpha(t_j) \colon t_j\in I,\, \lambda_j\ge 0, \, \sum_{j=1}^m \lambda_j=1,\, t=\sum_{j=1}^m \lambda_j t_j\right\}.
\]

Let $\varphi$ be an Orlicz function and put $M_\varphi$ to denote the separable part of the Lorentz space $L_\varphi([0,1])$. If $\liminf_{x\to\infty} \varphi(x)/x=0$, then the Banach envelope of $M_\varphi$ is zero, while if $\liminf_{x\to\infty} \varphi(x)/x>0$ then the Banach envelope of $M_\varphi$ is $M_\psi$ via the inclusion map, where the convex Orlicz function $\psi$ is an extension of $\widehat{\varphi|_{[c,\infty)}}$, and $c>0$ is such that $\inf_{x\ge c} \varphi(x)/x>0$. Subsequently, Kalton obtained the analogous result for Orlicz sequence spaces \cite{Kalton1977b}. To be more specific, Kalton proved that if $\ell_\varphi$ is separable, i.e., $\ell_\varphi$ and its separable part $m_\varphi$ coincide, then its Banach envelope is $\ell_\psi$ under the inclusion map, where the convex Orlicz function $\psi$ is obtained by extension of $ \widehat{\varphi|_{[0,1]} }$. This also inspired further research in \cite{DN1982}, where a representation of the Banach envelope of non-locally convex Orlicz sequence spaces was obtained without the assumption of separability. In this article, Drewnowski and Nawrocki proved that if $\psi$ is as above, then $\ell_\psi$ is the Banach envelope of $\ell_\varphi$ if and only if $\varphi(2t) \le c\max\{ \varphi(t), \psi(bt)\}$ for small $t$, and some constants $b,c>0$. They also proved that, in general, the Banach envelope of $\ell_\varphi$ is represented by the closure of $J(\ell_\varphi)$ in $\ell_\psi\oplus \ell_\varphi/m_\varphi$, where the envelope map $J$ is given by $J(f)=(f,f+m_\varphi)$ for all $f\in\ell_\varphi$. Other more recent articles working specifically on Banach envelopes include, e.g., \cites{AABW2019, CW2017, Drewnowski1988, KamLin2014, KamMas2007, KaltonSukochev2008}.

Although Banach envelopes have played a key role in the understanding of the geometry of nonlocally convex quasi-Banach spaces, in our opinion a systematic treatment of the concept is missing. The lack of a reliable reference makes the authors impose different assumptions in order to feel at ease working with them. For instance, the Banach envelope of a quasi-Banach space is often times thought of as the ``smallest'' Banach space containing $X$. In doing so we are taking for granted that the envelope map $J_{X}\colon X\to \widehat X$ is one-to-one (see, e.g., \cite{Shapiro1976}), but as the examples above show, this is not always the case. Think of the space $L_p$ for $p<1$. One could argue that this is a pathological case, and that what happens is simply that $L_p$ has no Banach envelope. In fact, in all the examples cited above, the envelope map happens to be either zero or one-to-one. In spite of that, we encounter intermediate situations like the spaces $\ell_p\oplus L_p$ ($0<p<1$) whose dual is neither trivial nor separating, i.e., whose envelope map is neither zero nor one-to-one. Of course, one could keep on considering this as a pathological example, or perhaps too simple to start paying attention to quasi-Banach spaces whose envelope map is not one-to-one.

However, there are quasi-Banach spaces whose Banach envelope provides valuable information about the linear structure of the space despite the fact that the envelope map is not one-to-one (or it is unclear whether it is one-to-one or not). A classical example is the Lorentz space $L_{1,\infty}$ over $[0,1]$. It is known that even though every linear functional on $L_{1,\infty}$ vanishes on simple functions, the dual of $L_{1,\infty}$ is nontrivial (see \cite{CwiSag1971}). The Banach envelope of $L_{1,\infty}$ was described in \cite{CF1980}. More recently, the authors of \cite{AACD2019} used the Banach envelopes of the Lipschitz free $p$-spaces $\F_{p}(\RR^2)$ and $\F_{p}(\RR)$ for $0<p<1$ to show that they are not linearly isomorphic. And the same to tell apart $\F_p(\RR)$ from $\ell_p$ for $0<p<1$. And all that without knowing whether the envelope map of $\F_{p}(\RR^d)$ is one-to-one or not!

Of course, assuming the point-separation property is essential in some situations. Suppose you want to determine whether a given Banach space $Z$ can be the Banach envelope of some quasi-Banach space $X$. Since the Banach envelope of $Z\oplus L_p$ for $0<p<1$ is $Z$ via the canonical projection, the problem is nonsensical (or trivial) unless we impose that $X$ is non-locally convex and has the point-separation property. The answer to this question resides with the finite-dimensional structure of the space $Z$. Kalton proved in \cite{Kalton1978} that if $X$ is a non-locally convex quasi-Banach space with the {point-separation} property then its Banach envelope has no type $p$ for any $p > 1$. It follows that the spaces $\ell_{p}$ and $L_{p}$ for $1 < p < \infty$, or the Orlicz spaces $\ell_{\varphi}$ and $L_{\varphi}$ with the lower Orlicz-Matuszewska index $\alpha_{\varphi}>1$, cannot be Banach envelopes of $X$.

Summing up, we must pay attention to the point-separation property of quasi-Banach spaces. Indeed, finding sufficient conditions that guarantee that the envelope map is one-to-one has both theoretical and practical interest. To substantiate our claim, let us explain a situation that occurs in applications. Some quasi-Banach spaces $X$ (such as Lipschitz free $p$-spaces for $0<p<1$) are naturally constructed from a non-complete quasi-normed space $X_0$ that is dense in $X$. This definition leads to a precise description on how $X_0$ embeds into $\widehat X$ via $J_X$, but not on how $X$ does! Thus, if we want to know that the character of the members of $X_{0}$ inside $\widehat X$ is preserved in passing to its completion, it is important to know whether we can transfer properties from $J_X|_{X_0}$ to $J_X$. In particular, the following question arises:

\begin{question}\label{PietschPhilQEnv}
Let $X$ be a quasi-Banach space and let $X_{0}$ be a dense subspace of $X$. Assume that $X^{*}$ separates the points of $X_{0}$. Does it follow that $X^{*}$ separates the points of $X$? Equivalently, assuming that $J_X|_{X_0}\colon X_{0}\to \widehat X$ is one-to-one, is $J_{X}\colon X\to \widehat X$ one-to-one?
\end{question}

The authors were glad to learn about other colleagues concern with this kind of problem. For instance, in his study of the Banach envelope of $\ell_{1,\infty}$ (\cite{Pietsch2009}),
 Pietsch raised the following question, which he elevated to the category of `philosophical':
\begin{question}[\cite{Pietsch2009}*{p.~214}]\label{PietschPhilQ}
Let $X_0$ and $Z$ be normed spaces. Suppose that $Z$ is complete and that $T\colon X_0\to Z$ is a one-to-one linear operator. Is the extension of $T$ to the completion of $X_0$ one-to-one?
\end{question}

Although the answer to this question is negative in general as Pietsch showed, we can solve the problem in the positive in some important situations. A famous example also brought out by Pietsch is the canonical embedding, say $T$, of the `algebraic' tensor product $X^*\otimes X$ into $\LL(X,X)$, where $X$ is a Banach space. Grothendieck \cite{Grothendieck1955}*{\S{I}, p.~165} proved that the extension of $T$ to the projective tensor product $X^*\widehat\oplus X$ is one-to-one if and only if $X$ has the approximation property (AP for short). Hence, if $X$ has the AP, we can regard the vectors in $X^*\widehat\oplus X$ as operators.

Question~\ref{PietschPhilQ} makes perfect sense for quasi-normed spaces too. In \cite{Vogt1967} we find an example of this (extended) question that lies at the core of the theory of nonlocally convex quasi-Banach spaces. Given a measure space $(\Omega,\mu)$ and a quasi-Banach space $X$, let $L_1(\mu,X)$ be the quasi-Banach space of all strongly measurable functions $f\colon\Omega\to X$ such that $\int_\Omega \Vert f\Vert \, d\mu<\infty$. Vogt \cite{Vogt1967} endowed the tensor product $L_1(\mu)\otimes X$ with a suitable quasi-norm so that both natural maps $T\colon L_1(\mu)\otimes X \to L_1(\mu,X)$ and $\II\colon L_1(\mu)\otimes X \to X$ are bounded. Then he proved that the extension of $T$ to the completion, $L_{1,V}(\mu,X)$, of $( L_1(\mu)\otimes X,\Vert \cdot\Vert)$ is one-to-one. This way the elements of $L_{1,V}(\mu,X)$ `are functions' and, so, the extension of $\II$ to $L_{1,V}(\mu,X)$ becomes an integral for such functions.

Question~\ref{PietschPhilQEnv} might be considered as well as a particular case of Question~\ref{PietschPhilQ}. Since quasi-Banach spaces whose envelope map is not one-to-one can be considered as 'pathological', one could be inclined to believe that the problem should have a positive answer. To our surprise, the experts on the subject were unable to either write down a proof or show us an example that disproved this conjecture. Our aim in this note is to fill this gap in the literature about Banach envelopes by solving in the negative Question~\ref{PietschPhilQEnv}.

\subsection*{Background results}
We conclude the introduction by making a few remarks which we hope will assist the non-specialist reader. 

The orthogonal sets associated with a quasi-Banach space $X$ are defined as follows. Given a subset $B$ of $X^*$ we put
\[
B^\perp=\{ x\in X \colon x^*(x)=0 \quad \forall x^*\in B\}.
\]
Similarly, given a subset $A$ of $X$, we put
\[
A^\perp=\{ x^*\in X^* \colon x^*(x)=0 \quad \forall x\in A \}.
\]
We can naturally identify $A^\perp$ with the dual space of the quotient space $X/V$, where $V$ is the closed subspace of $X$ spanned by $A$. To be precise, the dual map $Q^*\colon (X/V)^*\to X^*$ of the canonical map $Q\colon X\to X/V$ is an isometric embedding whose range is $V^\perp=A^\perp$. As a consequence we obtain
\begin{equation}\label{eq:perpep}
((X/V)^*)^\perp={V^{\perp\perp}}/{V}.
\end{equation}

Suppose that $Z$ is another quasi-Banach space. A bounded linear map $Q\colon X\to Z$ is said to be a \emph{quotient map} if $Q(B_X)$ contains a neighbourhood of zero. Note that $Q$ is a quotient map if and only if the canonical map $\widetilde Q\colon X/\Ker(Q)\to Z$ is an isomorphism. In case that $\widetilde{Q}$ is a $\lambda$-isomorphism, we say that $Q$ is a $\lambda$-quotient map. By the Open Mapping theorem, the continuous linear map $Q$ is a quotient map if and only if it is onto. Hence, $Q$ is a $\lambda$-quotient map if and only if there is short exact sequence
\[
0 \rightarrow Y \overset{S} \rightarrow X \overset{Q}\rightarrow Z \rightarrow 0,
\]
for some quasi-Banach space $Y$ and some linear map $S$ with $\Vert S\Vert\le\lambda/\Vert Q\Vert$.

\begin{proposition}\label{lem:quotientenvelope}
Suppose $X$ and $Z$ are quasi-Banach spaces. Let $Q\colon X\to Z$ be a $\lambda$-quotient map for some $\lambda\ge 1$. Then $\widehat{Q} \colon \widehat{X} \to\widehat{Z}$ is a $\lambda$-quotient map. Moreover,
\[
\Ker(\widehat{Q})=\overline{ J_X(\Ker(Q))}.
\]
\end{proposition}

\begin{proof} Set $V=\Ker(Q)$. Since $\widehat{Q} \circ J_X=J_Z\circ Q$, we have $\widehat{Q}(x)=0$ for all $x\in W=\overline{ J_X(V)}$. Hence, if 
\[
P\colon \widehat{X}\to U:={\widehat{X}}/{W}
\]
is the canonical map, there is $T\colon U\to \widehat{Z}$ such that $T\circ P=\widehat{Q}$. By the Hahn-Banach theorem, it suffices to prove that $T^*\colon (\widehat{Z})^*\to U^*$ is a $\lambda$-isomorphism. 

The space $U^*$  is isometrically isomorphic to $W^\perp$. The  same can be said about the spaces $(\widehat{Z})^*$ and $Z^*$, and about  $(\widehat{X})^*$ and $X^*$. Besides, the latter isometry restricts to a natural identification between $W^\perp$ and $V^\perp$. Thus, we need only prove that the map $T'\colon Z^*\to V^\perp$ given by
\[
z^*\circ J_Z\mapsto z^*\circ T\circ P\circ J_X, \quad z^*\in (\widehat{Z})^*,
\]
is a $\lambda$-isomorphism. Since $T\circ P\circ J_X=J_Z\circ Q$, the map $T'$ is given by $T'(z^*)=z^*\circ Q$ for all $z^*\in Z^*$. Since $Q$ is a $\lambda$-quotient map with null space $V$, we are done.
\end{proof}

\begin{corollary}\label{cor:1} Let $X$ and $Z$ be quasi-Banach spaces, and suppose that $Q\colon X \to Z$ is a quotient map. Then, $\widehat{Q}\colon\widehat{X}\to \widehat{Z}$ is an isomorphism if and only if
 $\Ker(Q)\subseteq (X^*)^\perp$.
\end{corollary}
\begin{proof} By Proposition~\ref{lem:quotientenvelope}, $\widehat{Q}$ is an isomorphism if and only the closure of $W:=J_X(\Ker(Q))$ is null. Clearly, the closure of $W$ is null if and only if $W$ is null. In turn, $W=\{0\}$ if and only if $\Ker(Q)\subseteq\Ker(J_X)$. Since $\Ker(J_X)=(X^*)^\perp$, we are done.
\end{proof}

 Our next result is just a re-formulation of the ``if'' part of Corollary~\ref{cor:1}.
\begin{corollary}\label{cor:2} Let $X$ and $Z$ be quasi-Banach spaces and suppose that $Q\colon X \to Z$ is a quotient map with $\Ker(Q)\subseteq N:=(X^*)^\perp$. Then, if $\lambda\ge 1$.
\begin{enumerate}[(i)]
\item $\widehat{Z}$ is a $\lambda$-isomorphic representation of the Banach envelope of $X$ via $Q$, and
\item $\widehat{X}$ is a $\lambda$-isomorphic representation of the Banach envelope of $Z$, via the unique linear bounded map $J\colon Z\to \widehat{X}$ with $J\circ Q=J_X$. 
\end{enumerate}
In particular, $\widehat{X}$ is an isometric representation of the Banach envelope of $X/N$ via the inclusion map.
\end{corollary}

The following proposition is an immediate consequence of Corollary~\ref{cor:2}, and the subsequent lemma a partial converse of it.
\begin{proposition}\label{prop:XNPointSeparation} Suppose $X$ is a quasi-Banach space, and let $N=(X^*)^\perp$. Then $X/N$ has the point-separation property.
\end{proposition}

\begin{lemma}\label{lem:XVPointSeparation}Let $V$ be a closed subspace of a quasi-Banach space $X$. Suppose that $X/V$ has the point-separation property. Then $(X^*)^\perp\subseteq V$.
\end{lemma}

\begin{proof}We have $\{0\} = ((X/V)^*)^\perp$. Then, by \eqref{eq:perpep},
$V^{\perp\perp}=V$. Finally, the inclusion $(X^*)^\perp\subseteq V^{\perp\perp}$ follows from the mere definition of orthogonal set.
\end{proof}

Let $0<p\le\infty$. A biorthogonal system $(\xx_n,\xx_n^*)_{n=1}^\infty$ is said to be \emph{$p$-Besselian} if the coefficient transform $\Fou\colon X\to\FF^\NN$ given by $\Fou(f)=(\xx_n^*(f))_{n=1}^\infty$ is a bounded operator from $X$ into $\ell_p$. The biorthogonal system $(\xx_n,\xx_n^*)_{n=1}^\infty$ is said to be semi-normalized if $0<\inf_n \Vert \xx_n\Vert$ and $\sup_n \Vert \xx_n\Vert<\infty$.

$(\ee_n)_{n=1}$ denotes the unit vector system of $\FF^\NN$, i.e., $\ee_n=(\delta_{k,n})_{k=1}^\infty$, where $\delta_{k,n}=1$ if $k=n$ and $\delta_{k,n}=0$ otherwise.

\begin{theorem}\label{thm:1BesBasisEnv}Let $X$ be a quasi-Banach space with a semi-normalized $1$-Besselian biorthogonal system $(\xx_n,\xx_n^*)_{n=1}^\infty$. Put
\[
Y=\overline{\spn(\xx_n\colon n\in\NN)}.
\]
Then the Banach envelope of $X$ can be represented by $\ell_1\oplus \widehat{X/Y}$ via the map
\[
x\mapsto (\Fou(x), J_{X/Y}(x+Y)).
\]
\end{theorem}

\begin{proof}Since $\sup_n \Vert J_X(\xx_n)\Vert\le \sup_n \Vert \xx_n\Vert<\infty$, there is a bounded linear operator $S\colon\ell_1\to \widehat{X}$ such that $S(\ee_n)= J_X(\xx_n)$ for all $n\in\NN$. We also have a bounded linear operator $\widehat{\Fou} \to \ell_1$ such that $\widehat{\Fou}\circ J_X=\Fou$. We infer that $\widehat{\Fou}\circ S=\Id_{\ell_1}$. Consequently, there is an isomorphism $P$ from $\widehat{X}$ onto $\ell_1\oplus \widehat{X}/S(\ell_1)$ such that
\[
P(x)= (\widehat{\Fou}(x), x+S(\ell_1)), \quad x\in \widehat{X}.
\]
Since $S(\ell_1)=\overline{J_X(Y)}$, applying Proposition~\ref{lem:quotientenvelope} yields an isomorphism
$T\colon \widehat{X}/S(\ell_1) \to \widehat{X/Y}$ such that
\[
T ( J_X(x)+S(\ell_1)) = J_{X/Y} (x+Y), \quad x\in X.
\]
Therefore, the Banach envelope of $X$ is $\ell_1\oplus \widehat{X/Y}$ via the map $J=(\Id_{\ell_1}, T)\circ R\circ J_X$.
Finally, for $x\in X$ we have
\[
J(x)= (\widehat{\Fou}(J_X(x)), T(J_X(x)+S(\ell_1)))=(\Fou(x), J_{X/Y}(x+Y)).\qedhere
\]
\end{proof}

\section{Point separation property and density}\label{sect:Main}
\noindent
Our proof in Theorem~\ref{thm:main1} below of the existence of quasi-Banach spaces whose envelope map is one-to-one on a dense subspace and yet do not have the point-separation property relies on the following proposiiton.
\begin{proposition}\label{prop:main} Let
\[
0 \rightarrow Y \overset{S} \rightarrow X \overset{Q}\rightarrow Z \rightarrow 0
\]
be a short exact sequence of quasi-Banach spaces. Suppose that $X$ is separable and that $Z$ is infinite-dimensional. Then there exists a dense subspace $X_0\subseteq X$ such that $Q$ is one-to-one on $X_0$.
\end{proposition}

\begin{proof} Without loss of generality we may assume that $S$ is an isometric embedding, i.e., $\Vert S(y)\Vert=\Vert y\Vert$ for all $y\in Y$, and that $Q$ is a $1$-quotient map, i.e., $\Vert z\Vert=\inf \{ \Vert x\Vert \colon Q(x)=z\}$ for all $z\in Z$. Let $\kappa\ge 1$ be such that $\Vert z+z'+z''\Vert \le \kappa (\Vert z\Vert +\Vert z'\Vert +\Vert z''\Vert)$ for all $z$, $z'$, $z''\in Z$.

Let $(z_n)_{n=1}^\infty$ be a linearly independent sequence in $Z$ such that $Z_0:=\spn(z_n\colon n\in\NN)$ is dense in $Z$. Pick an algebraic complement $Z_1$ of $Z_0$ and let $(z'_n)_{n=1}^\infty$ be a linearly independent sequence in $Z_1$ such that $\lim_n z'_n=0$. For each $n\in\NN$, let $x_n\in Q^{-1}(z_n)$ and $x'_n\in Q^{-1}( z'_n)$ be such that $\lim_n x'_n=0$.

Let $D\subseteq Y$ be countable and dense. Choose a sequence $y=(y_n)_{n=1}^\infty$ in $Y$ such that for every $k\in\NN$ and every $d\in D$ there is $n\ge k$ such that $y_n=d$. For instance, if $D=\{d_n \colon n\in\NN\}$, we could pick
\[
y=(d_1,d_1,d_2,\dots,d_1,\dots d_j,\dots, d_n,d_1,\dots,d_j,\dots,d_n,d_{n+1},\dots).
\]
Set
\[
X_0=\spn\left(\bigcup_{n=1}^\infty\{x_n, S(y_n)+x'_n\}\right).
\]
Let us see that $Q|_{X_0}$ is one-to-one. Suppose that $Q(x)=0$ for some $x\in X_0$. Then, if $(a_n)_{n\in A}$ and $(b_n)_{n\in B}$ are finite families of scalars such that
\[
x=\sum_{n\in A} a_n\, x_n + \sum_{n\in B} b_n\, (S(y_n)+\tilde x_n),
\]
applying $Q$ we obtain
\[
0=\sum_{n\in A} a_n\, y_n + \sum_{n\in B} b_n\, \tilde y_n.
\]
We infer that $a_n=b_k=0$ for every $n\in A$ and $k\in B$ and so $x=0$.

Let us next prove that that $X_0$ is a dense subspace of $X$. Let $x\in X$ and $\varepsilon>0$. Pick $(a_n)_{n=1}^m$ such that
\[
\left\Vert Q(x)-\sum_{n=1}^m a_n \, z_n\right\Vert < \frac{ \varepsilon}{3\kappa}.
\]
Since
\[
Q\left(x-\sum_{n=1}^m a_n \, x_n\right)= Q(x)-\sum_{n=1}^m a_n \, y_n,\] there is $y\in Y$ such that \[
\left\Vert -S(y)+ x-\sum_{n=1}^m a_n \, x_n\right\Vert <\frac{ \varepsilon}{3\kappa}.
\]
Let $y'\in D$ be such that $\Vert y- y'\Vert \le 3^{-1} \kappa^{-1} \varepsilon$. Now let $r\in\NN$ be such that $y_r=y'$ and $\Vert x'_r\Vert\le 3^{-1} \kappa^{-1} \varepsilon$. We have
\[
x':= S(y_r)+x'_r+ \sum_{n=1}^m a_n \, x_n\in X_0,
\]
and
\begin{align*}
\left\Vert x-x'\right\Vert
&\le\kappa\left( \Vert S(y)-S(y_r)\Vert+\Vert x'_r\Vert +\left\Vert -S(y)+ x-\sum_{n=1}^m a_n \, x_n\right\Vert\right)\\
&=\kappa\left( \Vert y-y_r\Vert+\Vert x'_r\Vert +\left\Vert -S(y)+ x-\sum_{n=1}^m a_n \, x_n\right\Vert\right)
< \varepsilon.\qedhere
\end{align*}
\end{proof}

\begin{theorem}\label{thm:main1}Let $X$ be a separable quasi-Banach space with $N=(X^*)^\perp\neq \{ 0\}$ and $X/N$ infinite-dimensional. Then:
\begin{enumerate}[(i)]
\item\label{item:1:main1} $J_X$ is not one-to-one, and
\item\label{item:2:main1} there exists a dense subspace $X_0\subseteq X$ such that $J_X$ is one-to-one on $X_0$.
\end{enumerate}
\end{theorem}

\begin{proof} \eqref{item:1:main1} is a consequence of Lemma~\ref{lem:2}. In order to prove \eqref{item:2:main1} we use Proposition~\ref{prop:main} to pick a dense subspace $X_0\subseteq X$ on which the canonical map $Q\colon X\to X/N$ is one-to-one. Then $I\circ Q$ is one-to-one on $X_0$, where $I\colon X/N\to\widehat{X}$ is the inclusion map. Since $J_X=I\circ Q$, we are done.
\end{proof}

Given a quasi-Banach space $X$, the closed linear span of the union of all subspaces of $X$ with trivial dual has trivial dual. Hence $X$ has a maximum subspace $V$ with trivial dual, and $V$ is closed. Following \cite{Kalton1981} we call $V$ the \emph{core} of $X$, and we put $V=\core(X)$. Clearly we have
\begin{equation}\label{eq:CorevsN}
\core(X)\subseteq (X^*)^\perp.
\end{equation}
In particular, if $X$ has the point-separation property its core is null.

\begin{example}\label{ex:main}
Let $Y$ be a nonzero separable quasi-Banach space with trivial dual and $Z$ be an infinite-dimensional space with the point-separation property. Suppose that $X$ is a twisted sum of $Y$ and $Z$, i.e., $X$ contains a closed subspace $V$ isomorphic to $Y$ with $X/V$ isomorphic to $Z$. Then, $V\subseteq \core(X)$ and, by Lemma~\ref{lem:XVPointSeparation}, $(X^*)^\perp\subseteq V$. Taking into account the inclusion \eqref{eq:CorevsN} we obtain
\[
V=\core(X)= (X^*)^\perp.
\]
Therefore, Theorem~\ref{thm:main1} applies to $X$. Note that $X$ is a twisted sum of $Y$ and $Z$ if and only if $X$ is isomorphic to $Y\oplus_F Z$ for some quasi-linear map $F\colon Z\to Y$ (see \cite{KaltonPeck1979}*{Theorem 2.4}). In particular Theorem~\ref{thm:main1} applies to the direct sum $Y\oplus Z$. To give a definite example, let us mention that although the envelope map of $L_p\oplus \ell_p$, $0<p<1$, is not one-to-one, it is one-to-one on a dense subspace.
\end{example}

If $X$ is as in Example~\ref{ex:main}, then its core is nonzero. We close by exhibiting that we can also apply Theorem~\ref{thm:main1} to quasi-Banach spaces with null core. Note that, surprisingly, this yields a quasi-Banach space $X$ with $\core(X)\subsetneq (X^*)^\perp$.

\begin{example}
In \cite{Ribe} (see also \cite{KPR1984}*{\S5.4}) Ribe constructed a quasi-Banach space $R$ with a one-dimensional subspace $L\subseteq R$ such that $R/L$ is isomorphic to $\ell_1$, i.e., $R$ is a twisted sum of $\FF$ and $\ell_1$, and $L$ is uncomplemented in $R$. In particular $R/L$ has the point-separation property. By Lemma~\ref{lem:XVPointSeparation}, $(R^{\, *})^\perp\subseteq L$ and, by Lemma~\ref{lem:FXUncomplemented}, $L\subseteq (R^{\,*})^\perp$. Therefore, $(R^{\, *})^\perp=L$. As far as the core is concerned, since $L^*\neq\{0\}$, applying the inclusion \eqref{eq:CorevsN} yields $\core(R)\subsetneq L$. Consequently, $\core(R)=\{0\}$.

By Lemma~\ref{lem:zero}, $J_X$ is not one-to-one. Since $R/ (R^{\,*})^\perp$ is isomorphic to $\ell_1$, it is infinite-dimensional. Therefore, Theorem~\ref{thm:main1} applies to $R$. We conclude by identifying the Banach envelope of $R$. The properties of the Ribe space yield a quotient map $Q\colon R\to \ell_1$ with $\Ker(Q)=L$. By Corollary~\ref{cor:2}, $\ell_1$ is a representation of the Banach envelope of $R$ via $Q$.
\end{example}


\begin{bibsection}
\begin{biblist}

\bib{AlbiacAnsorena2015}{article}{
author={Albiac, F.},
author={Ansorena, J.~L.},
title={On the mutually non isomorphic $\ell_p(\ell_q)$ spaces, II},
journal={Math. Nachr.},
volume={288},
date={2015},
number={1},
pages={5--9},
}

\bib{AlbiacAnsorena2017+}{article}{
author={Albiac, F.},
author={Ansorena, J.~L.},
title={Isomorphic classification of mixed sequence spaces and of Besov spaces over $[0,1]^d$},
journal={Math. Nachr.},
volume={290},
date={2017},
number={8-9},
pages={1177--1186},
}

\bib{AABW2019}{article}{
author={Albiac, F.},
author={Ansorena, J.~L.},
author={Bern\'a, P.},
author={Wojtaszczyk, P.},
title={Greedy approximation for biorthogonal systems in quasi-Banach spaces},
journal={arXiv:1903.11651 [math.FA]},
}

\bib{AACD2019}{article}{
author={Albiac, F.},
author={Ansorena, J.~L.},
author={C\'uth, M.},
author={Doucha, M.},
title={Lipschitz free $p$-spaces for $0<p<1$},
journal={arXiv:1811.01265 [math.FA], To appear in Israel J. Math.},
date={2019},
}

\bib{Arino1987}{article}{
author={Ari\~{n}o, M.~A.},
title={Convexity and Banach envelope of the weak-$L_p$ spaces},
journal={Proc. Amer. Math. Soc.},
volume={100},
date={1987},
number={4},
pages={661--665},
}

\bib{AEP1988}{article}{
author={Ari\~{n}o, M.~A.},
author={Eldeeb, R.},
author={Peck, N.~T.},
title={The Lorentz sequence spaces $d(w,p)$ where $w$ is increasing},
journal={Math. Ann.},
volume={282},
date={1988},
number={2},
pages={259--266},
}

\bib{Cwikel1973}{article}{
author={Cwikel, M.},
title={On the conjugates of some function spaces},
journal={Studia Math.},
volume={45},
date={1973},
pages={49--55. (errata insert)},
}

\bib{CwiSag1971}{article}{
author={Cwikel, M.},
author={Sagher, Y.},
title={$L(p,\infty )$},
journal={Indiana Univ. Math. J.},
volume={21},
date={1971/72},
pages={781--786},
}

\bib{CF1980}{article}{
author={Cwikel, M.},
author={Fefferman, C.},
title={The canonical seminorm on weak $L^1$},
journal={Studia Math.},
volume={78},
date={1984},
number={3},
pages={275--278},
}

\bib{CRS2007}{article}{
author={Carro, M.~J.},
author={Raposo, J.~A.},
author={Soria, J.},
title={Recent developments in the theory of Lorentz spaces and weighted inequalities},
journal={Mem. Amer. Math. Soc.},
volume={187},
date={2007},
number={877},
pages={xii+128},
}

\bib{CW2017}{article}{
author={Czerwi\'{n}ska, M.~M.},
author={Kami\'{n}ska, A.},
title={Banach envelopes in symmetric spaces of measurable operators},
journal={Positivity},
volume={21},
date={2017},
number={1},
pages={473--492},
}

\bib{Day1940}{article}{
author={Day, M.~M.},
title={The spaces $L^p$ with $0<p<1$},
journal={Bull. Amer. Math. Soc.},
volume={46},
date={1940},
pages={816--823},
}

\bib{Drewnowski1988}{article}{
author={Drewnowski, L.},
title={Compact operators on Musielak-Orlicz spaces},
journal={Comment. Math. Prace Mat.},
volume={27},
date={1988},
number={2},
pages={225--232},
}

\bib{DN1982}{article}{
author={Drewnowski, L.},
author={Nawrocki, M.},
title={On the Mackey topology of Orlicz sequence spaces},
journal={Arch. Math. (Basel)},
volume={39},
date={1982},
number={1},
pages={59--68},
}

\bib{DRS1969}{article}{
author={Duren, P.~L.},
author={Romberg, B.~W.},
author={Shields, A.~L.},
title={Linear functionals on $H^{p}$ spaces with $0<p<1$},
journal={J. Reine Angew. Math.},
volume={238},
date={1969},
pages={32--60},
}

\bib{Grothendieck1955}{article}{
author={Grothendieck, A.},
title={Produits tensoriels topologiques et espaces nucl\'{e}aires},
language={French},
journal={Mem. Amer. Math. Soc.},
volume={16},
date={1955},
pages={Chapter 1: 196 pp.; Chapter 2: 140},
}

\bib{Hunt1966}{article}{
author={Hunt, R.~A.},
title={On $L(p,\,q)$ spaces},
journal={Enseign. Math. (2)},
volume={12},
date={1966},
pages={249--276},
}

\bib{Kalton1977a}{article}{
author={Kalton, N.~J.},
title={Compact and strictly singular operators on Orlicz spaces},
journal={Israel J. Math.},
volume={26},
date={1977},
number={2},
pages={126--136},
}

\bib{Kalton1977b}{article}{
author={Kalton, N.~J.},
title={Orlicz sequence spaces without local convexity},
journal={Math. Proc. Cambridge Philos. Soc.},
volume={81},
date={1977},
number={2},
pages={253--277},
}

\bib{Kalton1978}{article}{
author={Kalton, N.~J.},
title={The three space problem for locally bounded $F$-spaces},
journal={Compositio Math.},
volume={37},
date={1978},
number={3},
pages={243--276},
}

\bib{Kalton1981}{article}{
author={Kalton, N.~J.},
title={Curves with zero derivative in $F$-spaces},
journal={Glasgow Math. J.},
volume={22},
date={1981},
number={1},
pages={19--29},
}

\bib{Kalton1986}{article}{
author={Kalton, N.~J.},
title={Banach envelopes of nonlocally convex spaces},
journal={Canad. J. Math.},
volume={38},
date={1986},
number={1},
pages={65--86},
}

\bib{KaltonPeck1979}{article}{
author={Kalton, N.~J.},
author={Peck, N.~T.},
title={Twisted sums of sequence spaces and the three space problem},
journal={Trans. Amer. Math. Soc.},
volume={255},
date={1979},
pages={1--30},
}

\bib{KPR1984}{book}{
author={Kalton, N.~J.},
author={Peck, N.~T.},
author={Roberts, J.~W.},
title={An $F$-space sampler},
series={London Mathematical Society Lecture Note Series},
volume={89},
publisher={Cambridge University Press, Cambridge},
date={1984},
pages={xii+240},
}

\bib{KaltonSukochev2008}{article}{
author={Kalton, N.~J.},
author={Sukochev, F.~A.},
title={Symmetric norms and spaces of operators},
journal={J. Reine Angew. Math.},
volume={621},
date={2008},
pages={81--121},
}

\bib{KamMas2007}{article}{
author={Kami\'{n}ska, A.},
author={Masty\l o, M.},
title={Abstract duality Sawyer formula and its applications},
journal={Monatsh. Math.},
volume={151},
date={2007},
number={3},
pages={223--245},
}

\bib{KamLin2014}{article}{
author={Kami\'{n}ska, A.},
author={Lin, P.-K.},
title={Banach envelopes of $p$-Banach lattices, $0<p<1$, and Ces\`aro spaces},
journal={Funct. Approx. Comment. Math.},
volume={50},
date={2014},
number={2},
pages={297--306},
}

\bib{Markus1964}{article}{
author={Markus, A.~S.},
title={Eigenvalues and singular values of the sum and product of linear operators},
language={Russian},
journal={Uspehi Mat. Nauk},
volume={19},
date={1964},
number={4 (118)},
pages={93--123},
}

\bib{Mitjagin1964}{article}{
author={Mitjagin, B.~S.},
title={Normed ideals of intermediate type},
language={Russian},
journal={Izv. Akad. Nauk SSSR Ser. Mat.},
volume={28},
date={1964},
pages={819--832},
}

\bib{Peetre1974}{article}{
author={Peetre, J.},
title={Remark on the dual of an interpolation space},
journal={Math. Scand.},
volume={34},
date={1974},
pages={124--128},
}

\bib{Pietsch2009}{article}{
author={Pietsch, A.},
title={About the Banach envelope of $l_{1,\infty}$},
journal={Rev. Mat. Complut.},
volume={22},
date={2009},
number={1},
pages={209--226},
}

\bib{Ribe}{article}{
author={Ribe, M.},
title={Examples for the nonlocally convex three space problem},
journal={Proc. Amer. Math. Soc.},
volume={73},
date={1979},
number={3},
pages={351--355},
}

\bib{Shapiro1976}{article}{
author={Shapiro, J.~H.},
title={Mackey topologies, reproducing kernels, and diagonal maps on the Hardy and Bergman spaces},
journal={Duke Math. J.},
volume={43},
date={1976},
number={1},
pages={187--202},
}

\bib{Shapiro1977}{article}{
author={Shapiro, J.~H.},
title={Remarks on $F$-spaces of analytic functions},
conference={
title={Banach spaces of analytic functions},
address={Proc. Pelczynski Conf., Kent State Univ., Kent, Ohio},
date={1976},
},
book={
publisher={Springer, Berlin},
},
date={1977},
pages={107--124. Lecture Notes in Math., Vol. 604},
}

\bib{Vogt1967}{article}{
author={Vogt, D.},
title={Integrationstheorie in $p$-normierten R\"{a}umen},
language={German},
journal={Math. Ann.},
volume={173},
date={1967},
pages={219--232},
}

\bib{Woj1984}{article}{
author={Wojtaszczyk, P.},
title={$H_{p}$-spaces, $p\leq 1$, and spline systems},
journal={Studia Math.},
volume={77},
date={1984},
number={3},
pages={289--320},
}

\bib{Woj1997}{article}{
author={Wojtaszczyk, P.},
title={Uniqueness of unconditional bases in quasi-Banach spaces with applications to Hardy spaces. II},
journal={Israel J. Math.},
volume={97},
date={1997},
pages={253--280},
}

\end{biblist}
\end{bibsection}
\end{document}